\documentclass[11pt]{amsart}
\usepackage{amsthm}
\usepackage{amssymb}
\usepackage{latexsym}
\usepackage[all]{xy}
\usepackage{verbatim}
\usepackage{color}
\usepackage{mathtools}
\usepackage[T1]{fontenc}
\usepackage[utf8]{inputenc}
\usepackage[russian,english]{babel}

\newif\ifCM\CMtrue 

\newif\iffurther %
\furtherfalse

\newcommand\alg[1]{{{{{\bf{#1}}}}}}
\def\isom{{\;\cong\;}} 
\newcommand\subjectto{{\,|\ }}
\newcommand\tensor[1][{}]{{\,\otimes_{\!#1}\,}}
\newcommand\Tensor[1][{}]{{\,\bigotimes}}

\newcommand\Cref[1]{{Corollary~\ref{#1}}}

\newcommand\Pref[1]{{Proposition~\ref{#1}}}

\newcommand\Rref[1]{{Remark~\ref{#1}}}

\newcommand\Tref[1]{{Theorem~\ref{#1}}}
\newcommand\Fref[1]{{Figure~\ref{#1}}}
\newcommand\Sref[1]{{Section~\ref{#1}}}

\newcommand\Ssref[1]{{Subsection~\ref{#1}}}
\newcommand\M[1][2]{{\operatorname{M}_{#1}}}
\newcommand\rcfin[1][F]{{\operatorname{M}_{\mbox{\tiny{rc-fin}}}(#1)}}

\def\Z{\mathbb Z}
\def\CP{\mathcal{P}}

\def\N{{\mathbb{N}}}

\def\id{{\operatorname{id}}}
\def\ann{{\operatorname{ann}}}

\newcommand\Br[1]{{\operatorname{Br}^{#1}}}
\newcommand\Brp[2][p]{{{}_{#1}\!\operatorname{Br}^{#2}}}

\long\def\forget#1\forgotten{}
\long\def\hide#1\hidden{{\bf{[}}{\tiny{#1}}{\bf{]}}}
\newcommand{\set}[1]{{\left\{#1\right\}}}

\newcommand\suchthat{{\,:\ \,}}
\def\hra{{\hookrightarrow}}
\def\lam{{\lambda}}

\newcommand\eq[1]{{(\ref{#1})}}
\newcommand\ideal[1]{{\left<{#1}\right>}}
\newcommand\sg[1]{{\ideal{#1}}}

\def\co{{\,{:}\,}}
\def\ra{{\rightarrow}}

\def\sub{\subseteq}
\renewcommand{\span}{\operatorname{span}} %

\def\({\left(}
\def\){\right)}
\newcommand\End{{\operatorname{End}}}

\newcommand\divides{{\,|\,}}

\newcommand\CC[2][F]{\ \!\mathcal {C}^{#2}_{#1}} 
\newcommand\lcm{{\operatorname{lcm}}}
\newcommand\abs[1]{{\left|{#1}\right|}}

\newcommand\fd{{finite dimensional}}
\newcommand\fdcsa[1][] {{\fd{} central simple {#1}algebra}}

\newcommand\rank{{\operatorname{rank}}}

\newif\ifXY 
\XYtrue     

\ifXY
\usepackage{xy}
\fi
\ifXY
\xyoption{matrix}
\xyoption{arrow}
\xyoption{curve}
\fi

\newtheorem{thm}{Theorem}[section] 

\newtheorem{cor}[thm]{Corollary}
\newtheorem{defn}[thm]{Definition}

\newtheorem{exmpl}[thm]{Example}

\newtheorem{prop}[thm]{Proposition}

\newtheorem{rem}[thm]{Remark}

\newtheorem{construction}[thm]{Construction}

\newcommand\smat[4]{{\left(\begin{array}{cc}\!{#1}\!&\!{#2}\!\\[-0.05cm] \!{#3}\!&\!{#4}\!\end{array}\right)}}

\DeclareMathOperator{\Hom}{Hom} %

\newcommand\dlim[1][]{{\if!#1!{\underrightarrow{\lim}\;\!}\else{\varinjlim_{\:\!#1}}\fi}}   
\def\nn{{\bf{n}}}
\def\mm{{\bf{m}}}

\def\DM{{\mathcal{E}}}

\begin{document}

\title{The algebra of supernatural matrices}
\def\TBO{Tamar Bar-On}\def\SG{Shira Gilat}\def\UV{Uzi Vishne}\def\EM{Eliyahu Matzri}
\def\UVemail{vishne@math.biu.ac.il}
\def\TBOemail{tamarnachshoni@gmail.com}
\def\SGemail{shira.gilat@live.biu.ac.il}
\def\EMemail{elimatzri@gmail.com}
\address{Bar Ilan University, Ramat Gan 52900, Israel}

\author{\TBO}\email{\TBOemail}
\author{\SG}\email{\SGemail}
\author{\EM}\email{\EMemail}
\author{\UV}\email{\UVemail}
\thanks{The authors are partially supported by Israeli Science Foundation grants no.
	630/17 
	and 1994/20. 
}


\maketitle

\begin{abstract}
The algebra of supernatural matrices is a key example in the theory of locally finite central simple algebras, which studied in a previous paper of the authors (\cite{Local}). It is also a stand-alone object admits a rich study and various connections to other fields. The goal of this paper is to expose some new information about supernatural matrices, mainly in terms of the "inner" ways to identify such algebras, and their appearance as minimal solutions to equations of the form $M_n(A)\cong A$. Viewing a natural representation of this algebra, we show that supernatural matrices generalize both McCrimmon's deep matrices algebra and $m$-petal Leavitt path algebra. We also study their simple representations.
\end{abstract}

\section*{Introduction}

In the paper \cite{Local} the authors developed a comprehensive theory of algebras over a field $F$, which are locally $F$-central simple of finite dimension. Namely, every finitely generated subalgebra is contained in a finite dimensional simple subalgebra whose center is equal to~$F$.
Basic examples of such algebras are the locally matrix algebras, and in particular, the countably generated locally matrix algebras, which we refer as the \textit{supernatural matrices}- as they are defined by their "supernatural" degree. A \textbf{supernatural number} $\nn$, as presented in \cite{ribes2000profinite} in the context of the order of a profinite group, is a formal product of the form $\prod p^{\alpha_p}$, where $p$ ranges over the set of natural primes and where each $\alpha_p$ belongs to $\mathbb{N}\cup \{\infty\}$. Such numbers are also called \textbf{Steinitz numbers}, see, for example, \cite{bezushchak2020unital}. Locally martix algebras are simply direct (and hence injective) limits of matrix algebras over the same field, over arbitrary directed systems and with arbitrary maps. The supernatural matrix algebra of degree $\nn$ is defined as the direct limit over the set of natural divisors of $\nn$, with the diagonal maps $M_m(F)\to M_n(F)$, $x\to x\otimes 1=\operatorname{diag}(x)$, whenever $m|n$. In \cite{Local} it was proved that every countably generated locally matrix algebra is in fact a supernatural algebra, which is uniquely determined by its supernatural degree $\nn$, defined as the lcm of the degrees of the matrix subalgebras.   

Locally matrix algebras have already been studied in 1942 by Kurosh in his paper \cite{kurosh1942direct}, were they have been called "locally marticial rings", and in 1979 in the book \cite{goodearl1979neumann}, were they have been called "ultramatricial rings". Recently, there have been a lot of research regarding different aspects of locally matrix algebras, such as their primary decomposition and Morita equivalent, in the various of papers by Bezushchak and Bogdana. For example, see \cite{bezushchak2020morita,bezushchak2020primary,bezushchak2020unital,bezushchak2021derivations}. Supernatural matrices also have a key role when considered over $\mathcal{C}$ in the theory of $\mathcal{C}^*$ algebras. See, for example \cite{davidson1996c}.

In the current paper we study several new aspects of supernatural matrices, and point out their connection to other famous families of algebras. The paper is organized as follows: In Section 1  we describe their matrix units and present a simple module which allows a concrete description of the supernatural matrices inside the algebra of row- and column-finite matrices. In Section 2 we present their special rule in the class of locally finite central simple algebras. In Section 3 we study several forms of the equation  $\M[p](F) \tensor X \isom X$ and the connections between them. McCrimmon's algebra of deep matrices is the subject of Section 4. We show that supernatural matrices of the form $\M[m^{\infty}](F)$ are a  homomorphic image of the algebra of balanced deep matrices. In this sense, supernatural matrices~$\M[\nn](F)$  generalize deep matrices. A similar connection is presented to Leavitt path algebra (\Sref{sec:LPA}), leading to a new description of the Leavitt path algebra $L_F(1,m)$ as rectangular recurrent infinite matrices. We conclude with a description of elementary gradings of supernatural matrices. Eventually, in Section 6 we give some results regarding simple presentations of supernatural matrix algebras.

\medskip
We thank Be'eri Greenfeld for his helpful discussions on Leavitt path algebras.
\section{Identifying supernatural matrices}\label{s:snmat}
\subsection{Supernatural matrix units}\label{ss:units}

The algebra of $n$-by-$n$ matrices is generated by matrix units $e_{ij}$, subject to the relations $e_{ij}e_{i'j'} = \delta_{ji'} e_{ij'}$ and $\sum_{i=0}^{n-1} e_{ii} = 1$ (we find it convenient to start indexation from zero). In order to compare matrices of varying dimensions, we adopt a sequential notation for the indices.

Fix a strictly increasing sequence of natural numbers greater then 1 $n_1 \divides n_2 \divides n_3 \divides \cdots$, whose least common multiple is a prescribed supernatural number~$\nn$; and artificially set $n_0 = 1$. We present a natural number $k_0+k_1n_1+k_2n_2+\cdots+k_{t-1}n_{t-1}$ by the ``$\nn$-adic representation'' $k_{t-1}\cdots k_0$, where the digits satisfy $0 \leq k_i < n_{i+1}/n_i$. By definition each word $w = k_{t-1}\cdots k_0$ has a well-defined length $\abs{w} = t$, in spite of potential leading zeros.

Interpreting indices by the $\nn$-adic representation, $e_{k_{t-1}\cdots k_0,k_{t-1}'\cdots k_0'}$ are all the matrix units of $\M[n_t](F)$.

\begin{thm}
	The supernatural matrix algebra $\M[\nn](F)$ is generated (as a unital algebra) by the matrix units $e_{w,w'}$, where $w,w'$ are $\nn$-adic sequences of the same length, subject to three families of relations:
	\begin{equation}\label{dmrel1}
		e_{uu'}e_{ww'} = \delta_{u',w} e_{uw'}
	\end{equation}
	whenever $\abs{u} = \abs{u'} = \abs{w} = \abs{w'}$;
	\begin{equation}\label{dmrel0}
		e_{\emptyset\emptyset} = 1;
	\end{equation}
	where $\emptyset$ is the empty word; and
	\begin{equation}\label{dmrel3}
		e_{ww'} = \sum_{i=0}^{n_{t+1}/n_t-1} e_{iw,iw'}
	\end{equation}
	for every $w,w'$ of the same length $t \geq 0$. 
\end{thm}
\begin{proof}
	Induction on $t$, based on \eq{dmrel0} which is the case $t = 0$, \eq{dmrel3} implies
	\begin{equation}\label{dmrel2}
		\sum_{\abs{w} = t} e_{ww} = 1.
	\end{equation}
	
	Now the relations \eq{dmrel1} and \eq{dmrel2}, for indices of length $t$, define the matrix algebra $\M[n_t](F)$; and \eq{dmrel3} identifies its matrix units as sums of the respective matrix units in the $n_{t+1}/n_t$ blocks of size $n_t$ composing the matrices of size $n_{t+1}$, thereby defining the embedding $\M[n_t](F) \hra \M[n_{t+1}](F)$. The relations thus define the direct limit $\dlim \M[n_t](F) = \M[\nn](F)$.
\end{proof}

The relation \eq{dmrel3} shifts generators deeper into the algebra, which when plugged in \eq{dmrel1} provides a multiplication rule for generators of arbitrary lengths, namely
\begin{equation}\label{dmall}
	e_{u,u'} e_{w,w'} = \begin{cases}
		e_{vu,w'} & \mbox{if $w = vu'$} \\ e_{u,vw'} & \mbox{if $u' = vw$} \\ 0 & \mbox{if neither $u'$ nor $w$ is a tail of the other.} \end{cases}
\end{equation}

\subsection{A recognition lemma}

Let us provide another presentation of $\M[\nn](F)$, based on a matrix recognition theorem of Agnarsson, Amitsur and Robson:
\begin{rem}[{\cite[Theorem~7.5]{Lam:cm}}]\label{AAR}
	A ring is an $n \times n$ matrix ring over its center if and only if it has elements $a,b,c$ such that $b^n = 0$ and $ab^k+b^{n-k}c=1$, where $k \in \set{1,\dots,n-1}$ is arbitrary.
	
	Indeed, the relations are satisfied in $\M[n](F)$ by $\tilde{a} = e_{0,k}+\cdots+e_{n-k-1,n-1}$, $\tilde{b} = e_{1,0}+\cdots+e_{n-1,n-2}$ and $\tilde{c} = e_{0,n-k}+\cdots+e_{k-1,n-1}$.
\end{rem}

We need a similar presentation for a pair of matrix algebras:
\begin{prop}\label{AARme}
	Suppose the subalgebra $\M[n](F) \sub R$ is generated by $a,b,c$ as in \Rref{AAR} (where $n>1$ and $k$ is arbitrary). There is an intermediate matrix algebra $\M[n](F) \sub \M[nm](F) \sub R$ if and only if $b$ has an $m^{\operatorname{th}}$ root in~$R$.
\end{prop}
\begin{proof}
	Let $a,b,c$ satisfy $b^n = 0$ and $ab^k+b^{n-k}c = 1$. First assume $b = b'^m$ for some $b' \in R$. Then $b'^{nm} = b^n = 0$ and $ab'^{km}+b'^{(n-k)m}c = ab^k+b^{n-k}c=1$, so $a,b',c$ generate a copy of $\M[nm](F)$ by \Rref{AAR}, which contains $F[a,b'^m,c] \isom \M[n](F)$. Conversely, suppose $F[a,b,c]$ is contained in a copy of $\M[nm](F)$ in $R$. Then $b' = \sum_{i=1}^{n}\sum_{j=1}^{m} e_{jn-i,(j+1)n-i} + \sum_{i=1}^{n-1}e_{mn-i,n-i-1}$ is an $m$th root of $\tilde{b}$, and we may assume $a,b,c$ are $\tilde{a},\tilde{b},\tilde{c}$ by conjugation.
\end{proof}

\begin{thm}
	Let $\nn$ be a supernatural number. Suppose $\nn = \prod m_t$ where $m_t \in \N$, $m_1 > 1$. An $F$-algebra contains $\M[\nn](F)$ if and only if it has elements $a$, $c$ and $b_1,b_2,\dots$, such that (for $b_0 = 0$)
	\begin{eqnarray}
		b_{t}^{m_t} & = & b_{t-1} \quad (t\geq 1), \label{P1}\\
		a b_1 + b_1^{m_1-1} c &=& 1. \label{P2}
	\end{eqnarray}
\end{thm}
\begin{proof}
	Write $n_t = m_1 \cdots m_t$. If the algebra has $a,b_1,c$ satisfying the given relation, then it contains a copy of $\M[m_1](F)$ by \Rref{AAR}, which is contained in a chain $\M[m_1](F) \sub \M[m_1m_2](F) \sub \cdots$ by \Pref{AARme},  whose union is by definition $\M[\nn](F)$.
\end{proof}

(The series of relations \eq{P1} can be encoded by the informal expression $(b_t) = 0^{1/\nn}$).

\subsection{A faithful simple module}\label{ss:nadic}

Consider the vector space $F^{\omega}$ of all sequences over $F$ (strictly containing the vector space $F^{<\omega}$ of eventually finite sequences).

Fix a supernatural number $\nn$, with a sequence $1=n_0 \divides n_1 \divides n_2 \divides \cdots$ whose lcm is $\nn$ as before. For each $t\geq 0$, view $F^{n_t}$ as the space spanned by $\nn$-adic words $u$ of length $\abs{u} = t$. We embed $F^{n_t} \hra F^{n_{t+1}}$ by sending $u$ to the formal sum $\sum_i {iu}$. This system defines a direct limit $F^{\nn} = \dlim F^{n_t}$. We may identify $F^{\nn} \sub F^{\omega}$ as the subspace of sequences which are $n_t$-periodic for some $t$. The natural action of $\M[n_t](F)$ on $F^{n_t}$ extends to an action of $\M[\nn](F) = \dlim \M[n_t](F)$ on $F^{\nn} = \dlim F^{n_t}$ in the obvious manner. In fact, whenever $u,v,w$ are (finite) $\nn$-adic words and $\abs{u} =\abs{v}$, we have
\begin{equation}\label{Mmact}
	e_{u,v} w = \begin{cases} {w' u} & \mbox{if $w = w'v$},\\
		0 & \mbox{otherwise.}\end{cases}
\end{equation}

It is easy to see that $F^{\nn}$ is a faithful simple module of $\M[\nn](F)$. This action will be extended in \Sref{s:deep}. To put the action into a concrete form, let $\rcfin$ denote the algebra of row- and column-finite matrices. We say that a matrix $x \in \rcfin$ is {\bf{$\nn$-recurrent}} if there is a natural divisor $n \divides \nn$ and a matrix $a \in \M[n](F)$, such that $x = a \oplus a \oplus \cdots$.
\begin{cor}\label{recurr}
	$\M[\nn](F)$ can be identified as the subalgebra of $\rcfin$ composed of the $\nn$-recurrent matrices.
\end{cor}
\section{Locally finite dimensional central simple algebras}\label{sec:CC}

To put supernatural matrices in structural context, let us briefly discuss a new family of algebras, which is the subject of \cite{Local}, where we give complete proofs and more details. Some of these results can already been found in previous paper as stated in the introduction.

Let $F$ be a field.
Let $\CC{}$ denote the class of $F$-algebras which locally are central simple and finite dimensional. Here, an algebra is ``locally $\CP$'' if every finitely generated subalgebra is contained in a finitely generated subalgebra satisfying the property~$\CP$. Every algebra in $\CC{}$ is an injective direct limit of finite dimensional central simple algebras. Algebras in $\CC{}$ are central simple. since the class of von Neumann regular algebras is closed under direct limits, these algebras are immediately von Neumann regular. Observe that since matrix algebras over a field $F$ are finite-dimensional $F$-central simple, then locally matrix algebras are locally finite-dimensional central simple algebras.

Many of the basic ingredients in the theory of finite dimensional central simple algebras extend nicely to this new class. Let us start with the degree, defined in the finite dimensional theory as the square root of the dimension.
\begin{defn}
	Let $\alg{A}$ be an algebra in $\CC{}$. The {\bf{degree}} of $A$ is the $\lcm$ of the degrees of its \fdcsa[sub]{s} $A_0 \sub \alg{A}$.
\end{defn}

Equivalently, writing $\alg{A} = \dlim_{\gamma\in\Gamma} A_{\gamma}$ as a direct limit of finite dimensional central simple algebras, over an arbitrary directed set $\Gamma$, the degree is $\deg(\alg{A}) = \lcm \set{\deg{A_{\gamma}}}$. More generally, we have:
\begin{prop}
	Let $\alg{A} = \dlim[{\gamma}] \alg{A}_{\gamma}$ where $\alg{A}_{\gamma} \in \CC{}$. Then $\deg(\alg{A}) = \lcm\set{\deg(\alg{A}_\gamma)}$.
\end{prop}

It is easy to see that $\deg(\M[\nn](F)) = \nn$, thereby proving \Pref{S1}.
The class $\CC{}$ is closed to tensor product, which fits naturally with the degree:
\begin{prop}
	For $\alg{A}, \alg{B} \in \CC{}$, we have that $\deg(\alg{A}\tensor \alg{B})=\deg(\alg{A})\deg(\alg{B})$.
\end{prop}

\subsection{Algebras of countable dimension}\label{ss:31}

An algebra in $\CC{}$ is countably generated if and only if it can be presented as a direct limit over the first infinite ordinal $\omega$, namely $\alg{A} = \dlim_{n} A_n$, where the $A_n$ are finite dimensional central simple. We thus let $\CC{\omega}$ denote the class of countably generated algebras in $\CC{}$. For example, supernatural matrices are in $\CC{\omega}$.
\begin{prop}[{\cite[Prop.~9.6]{Local}}]\label{noarrow}
	For a sequence of finite dimensional central simple algebras $A_1,A_2,\dots$, all the direct limits $\dlim_{n} A_n$ are isomorphic (regardless of the morphisms $A_n \ra A_{n+1}$).
\end{prop}
The statement of \Pref{noarrow} is false for limits over arbitrary directed sets. Indeed one can form a system of matrix algebras with two compatible systems of morphisms, such that the direct limit is countably generated in one, and has no countable basis in the other.

\subsection{Primary decomposition}

Let $\set{A_\lam \suchthat \lam \in \Lambda}$ be an arbitrary set of algebras. The {\bf{infinite tensor product}} $\Tensor A_\lam$ is defined as the direct limit of the tensor products over finite subsets, with the obvious morphisms. If $\Lambda = \set{\lam_1,\lam_2,\dots}$ is countable, the limit is equal to $\dlim_{n} (A_{\lam_1} \tensor \cdots \tensor A_{\lam_n})$.

\begin{exmpl}
	Every supernatural matrix algebra is a countable tensor product of finite matrix algebras, as $\M[\nn](F) = \Tensor_{i < \omega} \M[m_i](F)$ where $\nn = \prod m_i$.
\end{exmpl}

Algebras in $\CC{\omega}$ have a primary decomposition:
\begin{thm}[{\cite[Theorem~13.22]{Local}}]\label{primarydec}
	Let $\alg{A} \in \CC{\omega}$. There is a unique decomposition $\alg{A} \isom \Tensor \alg{A}_p$ where $\deg(\alg{A}_p)$ is a (possibly infinite) $p$-power.
\end{thm}
This fact calls attention to the primary components, each containing the algebras of prime-power degree for some prime.
It follows from \cite[Prop.~18.3]{Local} that the algebra $\M[p^{\infty}](F)$ is the zero element in this monoid:
\begin{prop}\label{zero}
	For every algebra $\alg{A} \in \CC{\omega}$ with degree a power of $p$, we have that $\M[p^{\infty}](F) \tensor \alg{A} \isom \M[p^{\infty}](F)$.
\end{prop}

\subsection{The Brauer monoid}

Two finite dimensional central simple algebras are {\bf{Brauer equivalent}} if they coincide after  tensor product with (finite) matrix algebras. The equivalence classes form the Brauer group of the field,~$\Br{}(F)$. A similar notion can be extended to $\CC{\omega}$. In light of \Pref{zero}, allowing tensor product with arbitrary supernatural matrices will collapse the whole structure. On the other hand allowing only finite  matrices does not capture the infinite dimensionality of the algebras.

A supernatural number $\prod p^{\alpha_p}$ is {\bf{locally finite}} if all $\alpha_p < \infty$.
We say that algebras $\alg{A},\alg{B} \in \CC{}$ are {\bf{Brauer equivalent}} if they coincide after tensor product with locally finite supernatural matrix algebras. The equivalence classes compose a monoid $\Br{\omega}(F)$, whose unique maximal subgroup is the direct product of the $p$-primary components of the classical Brauer group. The direct sum is, obviously, the Brauer group itself.

Algebras of $p$-power degree are equivalent if they coincide after tensor product with finite matrix algebras of $p$-power degree. The equivalence classes compose the $p$-primary part of the Brauer monoid, $\Brp{\omega}(F)$. Whereas $\Br{}(F) \isom \bigoplus_p \Brp{}(F)$, we have that $\Br{\omega}(F) \isom \prod \Brp{\omega}(F)$. The zero element of $\Br{\omega}(F)$ is the absolute supernatural matrix algebra, $\M[\pi](F)$, where $\pi = \prod_p p^{\infty}$.

\subsection{Splitting fields}

We say that a field extension $K/F$ {\bf{splits}} an algebra $\alg{A} \in \CC{\omega}$ if $K \tensor[F] \alg{A}$ is a supernatural matrix algebra. For an algebraic extension $\alg{K}/F$, let the {\bf{supernatural dimension}} $\dim(\alg{K})$ be the lcm of the finite dimensions of subfields. We proved that if $\alg{D} \in \CC{\omega}$ is a division algebra, then a subfield $\alg{K} \sub \alg{D}$ splits $\alg{D}$ if and only if $\dim(\alg{K}) = \deg(\alg{D})$. In particular every division algebra splits, and in this sense every algebra in $\CC{\omega}$ is a ``form of matrices'', namely isomorphic to some $\M[\nn](F)$ after a suitable scalar extension.

\section{Cancellation of matrices}\label{s:matab}\label{sec:MpA=A}\label{Equation}

Cancellation is a classical topic in algebra. Lam's monograph \cite{Lam:cm} is devoted to module and matrix cancellation. The latter is concerned with the conditions under which
$\M[n](A) \isom \M[n](B)$
implies $A \isom B$. In this section we consider variations on a similar question: what algebras $A$ satisfy the isomorphism
\begin{equation}\label{eqn}
	\M[n](A) \isom A
\end{equation}
of unital algebras?

Supernatural matrix algebras are an obvious class of examples.
\begin{prop}
	Let $\nn$ be a supernatural number. For a natural number~$n$,
	$$\M[n](F) \tensor \M[\nn](F) \isom \M[\nn](F)$$
	if and only if $n^{\infty} \divides \nn$.
\end{prop}
\begin{proof}
	Follows from \Pref{mn=mn}, as $n^\infty\nn = \nn$ if and only if $n^\infty \divides \nn$.
\end{proof}
Therefore, for any set $P$ of natural primes, we have a supernatural matrix algebra $A$ for which $\M[p](F) \tensor A \isom A$ if and only if $p \in P$, namely the supernatural matrix algebra $A = \M[(p^*)^\infty](F)$ where $p^* = \prod_{p \in P}p$.

\smallskip

If $A \isom A^n$ as modules, then clearly $A \isom \End_A(A) \isom \End_A(A^n) \isom \M[n](A)$. We thus say that $A$ is a {\bf{modular solution}} to the equation~\eq{eqn}.  Interestingly, supernatural matrix algebras are a nonmodular solution, since they satisfy IBN (the Invariant Basis Number property, stating that $A^n \isom A^m$ implies $n = m$). Indeed, the direct limit of algebras with IBN (in particular, of finite dimensional algebras) is IBN as well.

\smallskip

\def\XI{{\mbox{$(c_1^*)$}}}
\def\XII{{\mbox{$(c_1)$}}}
\def\XIII{{\mbox{$(c_3)$}}}
\def\XIV{{\mbox{$(c_4^*)$}}}
\def\XV{{\mbox{$(c_4)$}}}
\def\XIIIs{{\mbox{$(c_2)$}}}
\def\XIVs{{\mbox{$(c_4^*)$}}}
\def\XVs{{\mbox{$(c_4^\#)$}}}
\def\XVII{{\mbox{$(c_3^*)$}}}
In light of primary decomposition (\Tref{primarydec}), let us focus on the case where the degree of the matrix algebra is a prime~$p$, in particular $\M[p](F)$ and $\M[p^\infty](F)$. Following \eq{eqn}, we consider an embedding of $\M[p](F) \tensor A$ or $\M[p^{\infty}](F) \tensor A$ into~$A$. The embedding may take three forms: as a  subalgebra; as a factor; or by an isomorphism. (We say that a subalgebra $B$ is a {\bf{factor}} in an algebra~$C$ if $C = B \tensor B'$ for some algebra $B'$; we write this as $B \divides C$). We are thus led to six embedding conditions, two of which are equivalent (see below). Adding the embedding of the matrix and supernatural matrix algebras themselves into $A$, we obtain seven conditions, that are presented with their logical dependencies in \Fref{ALL}. We denote the conditions by \XII, \dots, \XV, \XI, \dots, \XIV\ as in the diagram.

\begin{figure}
	$$\xymatrix@R=0.5pt@C=3pt{
		{} & \XI\ \ \M[p^{\infty}](F) \tensor A \isom A  \ar@/^4ex/@{->}[rddd] \ar@{->}[dd]  \\
		{} & {} & {\phantom{X}} \\
		{} & \XII\ \ \M[p](F) \tensor A \isom A \ar@{->}[dd] \\
		{} & {} & \XVII\ \ \M[p^{\infty}](F) \tensor A \hra A \ar@/^1ex/@{->}[lddd]  \\
		{} & \XIIIs\ \ \M[p](F) \tensor A \divides A \ar@{->}[dd] & \\
		{} & {} & {\phantom{X}} \\
		{} & \XIII\ \ \M[p](F) \tensor A \hra A \ar@{->}[dd] \\
		{} & {} & {\phantom{X}} \\
		{} & \XIV\ \ \M[p^{\infty}](F) \hra A \ar@{->}[dd]  \\
		{} & {} & {\phantom{X}} \\
		{} &  \XV\ \ \forall n: \M[p^n](F) \hra A
	}$$
	\caption{Conditions on unital algebras}\label{ALL}
\end{figure}
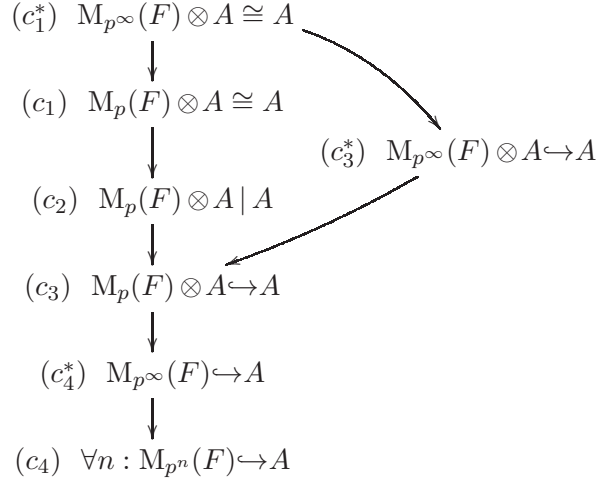

\smallskip
\begin{rem}\label{X1+}
	The algebra $A$ satisfies \XI\ if and only if it has $\M[p^{\infty}](F)$ as a factor. One direction is trivial. On the other hand if $A = \M[p^{\infty}](F) \tensor B$ for some~$B$, then $\M[p^{\infty}](F) \tensor A \isom (\M[p^{\infty}](F) \tensor \M[p^{\infty}](F)) \tensor B \isom \M[p^{\infty}](F) \tensor B = A$.
\end{rem}

In particular, $\M[p^{\infty}](F)$ solves \XI.

\begin{prop}
	We have that $$\XI \implies \XII \implies \XIIIs \implies \XIII \implies \XIV \implies \XV$$
	and $\XI \implies \XVII \implies \XIII$.
\end{prop}
\begin{proof}
	$\XI \implies \XII$ follows from $\M[p](F) \tensor \M[p^{\infty}](F) \isom \M[p^{\infty}](F)$.
	$\XII \implies \XIIIs$ and $\XIIIs \implies \XIII$ are trivial, and so are $\XIV \implies \XV$ and $\XI \implies \XVII \implies \XIII$.
	We thus prove $\XIII \implies \XIV$.
	We are given an embedding $\varphi_1 \co \M[p](F) \tensor A \hra A$.
	By induction, assume we fixed an embedding $\varphi_n \co \M[p^n](F) \tensor A \hra A$. We then have an embedding $$\varphi_{n+1} \co \M[p^{n+1}](F) \tensor A = \M[p^n](F) \tensor (\M[p](A) \tensor A) \hra \M[p^n](F) \tensor A \hra A,$$ defined by $\varphi_{n+1}(x \tensor t \tensor a) = \varphi_n(x \tensor \varphi_1(t \tensor a))$. In particular $\varphi_{n+1}(x \tensor 1 \tensor 1) = \varphi_n(x \tensor \varphi_1(1 \tensor 1)) = \varphi_n(x \tensor 1)$, so the sequence $\varphi_n \co \M[n](F) \ra A$ is compatible and defines an embedding of the direct limit $\M[p^{\infty}](F) \hra A$.
\end{proof}

The fact that $\XII \implies \XIV$ shows that~$\M[p^{\infty}](F)$ is a minimal solution to the equation $\M[p](A) \isom A$, in the sense that any solution contains a copy of~$\M[p^{\infty}](F)$. We compare this to Leavitt algebras in \Ssref{ssec:LPA} below.

\begin{exmpl}\label{matprod}
	Let $Q = \prod \M[p^n](F)$ be the direct product over all $n\geq 1$, and $J = \bigoplus \M[p^n](F)$ the direct sum, which is an ideal in $Q$. Although $Q$ does not satisfy \XV, 
	$Q/J$ does satisfy~\XII, by sending $x \tensor (a_1,a_2,\dots) + J \mapsto (0,x\tensor a_1, x \tensor a_2, \dots) + J$.
	
	And indeed, we have an embedding $\M[p^{\infty}](F) \hra Q/J$ by the compatible sequence of embeddings $\varphi_n \co a \mapsto (0,\dots,0,a,a\tensor 1,a\tensor 1\tensor 1,\dots)$.
\end{exmpl}

The condition \XI\ is strictly stronger than \XII:
\begin{prop}
	$\XII \ \, \not \!\!\! \implies \XI$.
\end{prop}
\begin{proof}
	We construct an algebra $A$, for which $\M[p](A)\isom A$ but $\M[p^\infty](F)\tensor A \not\cong A$. For an $F$-algebra $R$, let $\rcfin[R]$ denote the algebra of row and column finite matrices over $R$. The algebra of finite matrices is an ideal of $\rcfin[R]$, and as described in \Cref{recurr}, $\M[p^{\infty}](R)$ is the subring of $p^{\infty}$-recurrent matrices over $R$. Let $T(R)$ denote their sum.  Consider the algebra $A = T(F)$.
	It is easy enough to see that $\M[p](F) \tensor A \isom A = T(F)$, so $A$ satisfies~\XII.
	
	We also have that $\M[p^{\infty}](F) \tensor A = \M[p^{\infty}](F) \tensor T(F) = T(\M[p^{\infty}](F))$.
	Gene Abrams proved in \cite{GAb} that when $A$ and $B$ are simple, if $T(A) \isom T(B)$ and~$T$ is an intermediate algebra between finite and row-finite matrices, then $A$ and~$B$ are Morita equivalent. Since~$F$ and $\M[p^{\infty}](F)$ are not Morita equivalent, we have that $A = T(F) \not \cong T(\M[p^{\infty}](F)) = \M[p^{\infty}](F) \tensor A$, so $A$ does not satisfy~\XI.
\end{proof}

In particular $\XIV \ \, \not \!\!\! \implies \XI$: unlike standard matrix algebras, a subalgebra isomorphic to $\M[p^{\infty}](F)$ is not necessarily a factor.

\begin{rem}
	Each of the seven conditions is ``closed upwards'', in the sense that if $A$ satisfies it, then so does every algebra in which $A$ is a factor.
\end{rem}

\begin{rem}
	Raising to an infinite tensor power $A^{\tensor \infty} = \dlim A^{\tensor n}$ transfers solutions of some equation to solutions of more prominent ones, as follows. If $A$ satisfies \XII, then $A^{\tensor \infty}$ satisfies \XI.  If $A$ satisfies \XIIIs, namely $\M[p](F) \tensor A \tensor B \isom A$, then $A \tensor B^{\tensor \infty}$ satisfies \XII.
\end{rem}

\medskip

Let the $p$-{\bf{matrix degree}} of an algebra $B$ be the supremum of the set $\set{p^n \suchthat \M[p^n](F) \hra B}$.

\begin{prop}
	The conditions $\XI, \dots, \XV, \XVII$ are equivalent for infinite tensor products of 
	algebras of finite $p$-matrix degree. 
\end{prop}
In particular, the conditions are equivalent in the class $\CC{\omega}$ (\Ssref{ss:31}), where every algebra is a countable tensor product of finite dimensional central simple algebras.
\begin{proof}
	It suffices to prove $\XV \implies \XI$. Write $A = \Tensor A_{\lam}$, where $A_{\lam}$ are of finite matrix degree. By assumption there must be at least countably many algebras $A_{\lam}$ containing copies of $\M[p](F)$. Since finite matrix subalgebras are factors, each of those has the form $\M[p](F) \tensor A_{\lam}'$, and the infinite tensor product of the matrix factors is equal to $\M[p^\infty](F)$, thus showing that \XI\ holds by \Rref{X1+}.
\end{proof}

We say that a class ${\mathcal{M}}$ of algebras is {\bf{cancellable}} if $\M[p](F) \tensor A \isom \M[p](F) \tensor B$ implies $A \isom B$ for $A,B \in \mathcal{M}$, and {\bf{closed}} if $\M[p](F) \tensor A \in \mathcal{M}$ implies $A \in \mathcal{M}$.  For example, $\CC{}$ is cancellable and closed, essentially because the class of finite dimensional central simple algebras is cancellable and closed.
\begin{prop}
	Let $\mathcal{M}$ be a cancellable and closed class of algebras. Then $\XV \implies \XIV$ for the algebras in $\mathcal{M}$.
\end{prop}
\begin{proof}
	For every $n$, there is by assumption an embedding $\M[p^n](F) \ra A$. We can therefore decompose $A = \M[p^n](F) \tensor A_n$. Since $\mathcal{M}$ is closed, $A_n \in \mathcal{M}$ and $\M[p](F) \tensor A_{n+1} \in \mathcal{M}$.  Now $\M[p^n](F) \tensor A_n = A = \M[p^n](F) \tensor (\M[p](F) \tensor A_{n+1})$, so by cancellation, $A_n \isom \M[p](F) \tensor A_{n+1}$. We now construct a chain of finite matrix algebras, ordered by inclusion, inside $A$. Once $M_n = \M[p^n](F)$ is chosen, we fix a copy $\M[p](F) \isom Q_n \sub A_n$ and choose $M_{n+1} = M_n \tensor Q_n$. The limit $\dlim M_n = \M[p^\infty](F)$ is a subalgebra of $A$.
\end{proof}

The implication $\XII \implies \XIV$ has a stronger version for algebras $\alg{A} \in \CC{\omega}$: there is an isomorphism $\alg{A} \isom \alg{B} \tensor \alg{A}$ (with $\alg{B} \in \CC{\omega}$) if and only if $\alg{A}$ contains a copy of $\M[p^{\infty}](F)$ for some prime~$p$ (\cite[Cor.~19.6]{Local}).

\section{The algebra of deep matrices}\label{s:deep}

Our construction of supernatural matrices generalizes McCrimmon's algebra of deep matrices, introduced in \cite{DM1}.
\subsection{The homogeneous case}
Let $m$ be a natural number, and consider the supernatural power $\mm = m^{\infty}$. Thus $\mm$-adic words are nothing but numbers in base $m$ representation.

The algebra $\DM^m$ of {\bf{deep matrices on $m$ letters}} is generated by the elements $d_{u,v}$, where $u,v$ are arbitrary (finite) $\mm$-adic words, of any length; subject to the multiplication rule \eq{dmall} (we reverse the index sequences with respect to McCrimmon's notation, to conform with the standard decimal representation of the integers).

The algebra $\DM^m$ is graded over $\Z$ by $\deg(d_{u,v}) = \abs{u}-\abs{v}$. Matrix units of degree zero, namely $d_{u,v}$ with $\abs{u} = \abs{v}$, generate the degree zero component~$\DM^m_0$, termed ``balanced deep matrix algebra'' in \cite[Definition~4.2]{DM7}.

Let $\alg{V}_{\mm}$ be the vector space spanned by left-infinite $\mm$-adic sequences (thus $\dim(\alg{V}_{\mm}) = \aleph_0$). Let $\epsilon_{u,v} \in \End(\alg{V}_\mm)$ be the ``chopping and sewing heads'' operators
\begin{equation}\label{dmact}
	{\epsilon}_{u,v}\pi = \begin{cases} \pi' u & \mbox{if $\pi = \pi'v$},\\
		0 & \mbox{otherwise.}\end{cases}
\end{equation}
The ``Frankenstein action'' of $\DM^m$ on $\alg{V}_{\mm}$ is defined by sending $d_{u,v}$ to $\epsilon_{u,v}$; indeed, the $\epsilon_{u,v}$ satisfy \eq{dmall}.
We thus have a map $\Phi \co \DM^m \ra \End(\alg{V}_\mm)$. It is shown in~\cite{DM2} (when the letter set is finite, as in our case), that the kernel of~$\Phi$ is the unique nonzero ideal of $\DM^m$.

Let us explain how the supernatural matrix algebra $\M[\mm](F)$ fits into this picture.
The countably dimensional space $F^{\mm}$ (\Ssref{ss:nadic}) embeds into~$\alg{V}_\mm$ by sending a basis element~1$u$ to $\pi = {\cdots 000u}$. Now, because~\eq{dmact} is effectively identical with~\eq{Mmact}, the action of $\M[\mm](F)$ on $F^{\mm}$ extends to an action on~$\alg{V}_\mm$, defining an embedding $\M[\mm](F) \hra \End(\alg{V}_{\mm})$ by $e_{u,v} \mapsto \epsilon_{u,v}$. Indeed, the operators $\epsilon_{u,v}$ satisfy the relations~\eq{dmrel1}--\eq{dmrel2}.

\begin{figure}[!h]\label{Franky}
	$$\xymatrix{\DM^m_0 \ar@{^(->}[r] \ar@{->}[d]  & \DM^m \ar@{->}[d]^{\Phi} \\ \M[\mm](F) \ar@{^(->}[r] & \End(\alg{V}_{\mm})}$$
	\caption{Frankenstein action of deep matrices and supernatural matrices}
\end{figure}
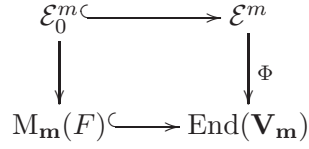

We now have a map $\DM^m_0 \ra \M[\mm](F)$ defined by sending a deep generator~$d_{u,v}$ to the respective supernatural matrix unit $e_{u,v}$; and the diagram in \Fref{Franky} is easily seen to commute. It follows that:
\begin{prop}\label{dmonto}
	The supernatural matrix algebra $\M[\mm](F)$, where $\mm = m^{\infty}$, is a homomorphic image of the algebra~$\DM_0^m$ of balanced deep matrices on~$m$ letters.
\end{prop}

\begin{rem}
	The action of the deep matrix algebra $\DM$ on $F^{\nn}$ is not faithful because $d_{\emptyset\emptyset}-\sum d_{ii}$ acts trivially. In comparison, the action of $\DM$ on $\bigoplus F^{m^t}$ is faithful by \cite[Theorem~2]{DM2}.
\end{rem}

\subsection{The general case}

Although deep matrices were only defined for a fixed set of letters (namely for $\mm = m^{\infty}$), most of the above construction extends easily to an arbitrary supernatural number $\nn$.

The algebra $\DM^\nn$ can be defined as above, and similarly has a $\Z$-grading. We thus have the balanced deep matrix algebra $\DM^\nn_0$. The space $\alg{V}_{\nn}$ is defined with no problem as well. However, the action of $\DM^\nn$ on $\alg{V}_{\nn}$ is not defined, because a generator $e_{u,v}$ with $\abs{u} \neq \abs{v}$ cannot shift an $\nn$-adic sequence, as the letter set depends on the position. Nevertheless, the action of $\DM^\nn_0$ {\it{is}} defined.
Indeed, since the relations \eq{dmall} provide the full multiplication table, it is easy to see that $\DM^\nn_0$ is generated by the balanced deep matrix units, subject only to \eq{dmall} (restricted to balanced generators), so $\Phi_0 \co \DM^\nn_0 \ra \End(\alg{V}_\nn)$ is well defined by $\Phi_0(d_{uv}) = \epsilon_{uv}$. We obtain the diagram in \Fref{deepF}, to be extended below.
\begin{figure}[!h]$$\xymatrix{\DM^\nn_0 \ar@{^(->}[r] \ar@{->}[d] \ar@{->}[dr]^{\Phi_0} & \DM^\nn  \\ \M[\nn](F) \ar@{^(->}[r] & \End(\alg{V}_{\nn})}$$
	\caption{Generalized deep matrices and supernatural matrices}
	\label{deepF}
\end{figure}
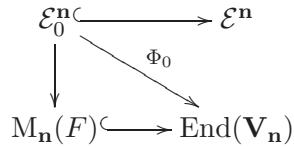

\section{Supernatural matrices and Leavitt algebras}\label{sec:LPA}

We offer a concrete realization of the Leavitt algebra as infinite matrices, based on their action on $\alg{V}_m$ and its connection with supernatural matrices. This connection can be exploited through nonstandard gradings, which leads to the characterization of elementary gradings for supernatural matrices.

For a recent review on Leavitt algebras see \cite{Lpa}. Let $m$ be a natural number. The Leavitt  algebra $L_F(1,m)$ of the $m$-petal is defined by the presentation
$$L_F(1,m) = F\sg{x_0,\dots,x_{m-1},y_0,\dots,y_{m-1} \subjectto y_i x_j = \delta_{ij}, \ \sum x_i y_i = 1}.$$

Fix the supernatural exponent $\mm = m^{\infty}$ as in \Sref{s:deep}. For an $\mm$-adic number $v = (i_{t-1},\dots,i_0)$, we write
$x_v = x_{i_{0}} \cdots x_{i_{t-1}}$ and $y_v = y_{i_{t-1}}\cdots y_{i_{0}}$ (notice the order inversion).

Since elements of the form $y_i x_j$ can be reduced, the Leavitt algebra is spanned by elements of the form $x_uy_v$ (where $u,v$ are of arbitrary length). The multiplication rule can easily be deduced from the defining relations:
\begin{equation}\label{xyrel}
	x_{u}y_{u'}\cdot x_{w}y_{w'} = \begin{cases}
		x_{vu}y_{w'} & \mbox{if $w=vu'$},\\
		x_{u}y_{vw'} & \mbox{if $u'=vw$},\\
		0 & \mbox{if neither $u'$ nor $w$ is a tail of the other.}\end{cases}
\end{equation}
The standard $\Z$-grading of $L_F(1,m)$ is given by $\deg(x_i) = - \deg(y_i) = 1$.
As already described in \cite{DM8}, we have:

\begin{prop}\label{DMtoLP}
	There is a degree-preserving projection $\DM^m \ra L_F(1,m)$ by $e_{u,v} \mapsto x_{u}y_v$.
\end{prop}
\begin{proof}
	The defining relations \eq{dmall} hold in $L_F(1,m)$ by \eq{xyrel}, and the map is clearly onto.
\end{proof}

Recall the space $\alg{V}_\mm$ and the operators $\epsilon_{uv}$ from \eq{dmact}. Verifying the defining relations, we observe that
there
is a map $L_F(1,m) \ra \End(\alg{V}_{\mm})$ by $x_i \mapsto \epsilon_{i\emptyset}$ and $y_i \mapsto \epsilon_{\emptyset i}$, factoring
the map $\Phi \co \DM^m \ra \End(\alg{V}_{\mm})$.

\begin{rem}
	The map takes $x_{u} \mapsto \epsilon_{u\emptyset}$ and $y_u \mapsto \epsilon_{\emptyset u}$; and therefore $x_{u}y_v \mapsto \epsilon_{u\emptyset}\epsilon_{\emptyset v} = \epsilon_{uv}$.
\end{rem}

We can thus extend \Fref{Franky}:
\begin{figure}[!h]
	$$\xymatrix@R=12pt@C=10pt{
		{} & \DM_0^m \ar@{^(->}[rr] \ar@{->}[dd]|{\hole} \ar@{->}[dl] 
		& {} & \DM^m \ar@{->}[dd]^{\Phi} \ar@{->}[dl]  & {}
		\\ L_F(1,m)_0 \ar@{->}[dr] \ar@{^(->}[rr] & {} & L_F(1,m)  \ar@{->}[dr] & {} & {}
		\\ {} & \M[\mm](F) \ar@{^(->}[rr] & & \End(\alg{V}_{\mm}) &
	}
	$$
	\caption{Maps between supernatural matrices, Leavitt algebra and deep matrices}\label{deepF+}
\end{figure}

The map of \Pref{DMtoLP} sends balanced deep matrices to the zero homogeneous component $L(1,m)_0$, spanned by elements $x_{\bar{u}}y_v$ where $\abs{u}=\abs{v}$. Again, the map $\DM_0 \ra \M[\nn](F)$ factors through $L(1,m)_0$.

\begin{thm}\label{same}
	For $\mm = m^{\infty}$ we have that $L(1,m)_0 \isom \M[\mm](F)$.
\end{thm}
\begin{proof}
	Since $L_F(1,m)$ is simple, the map $L_F(1,m) \ra \End(\alg{V}_{\mm})$ is injective, and so $L(1,m)_0 \ra \M[\mm](F)$ is also injective; but the latter map is onto by \Pref{dmonto}.
\end{proof}

This identification embeds $L(1,m)_0$ into $\rcfin$, as described in \Cref{recurr}. Let us generalize this to a description of the whole algebra $L(1,m)$.
For simplicity of the notation, let $\alg{V} = \alg{V}_\mm$. There is a natural isomorphism $\mu \co \alg{V} \ra \alg{V} \oplus \cdots \oplus \alg{V}$ ($m$ copies), defined by sending an infinite $\mm$-adic sequence $\pi = \pi' i$ to the vector whose $i$th entry is $\pi'$ (with zero everywhere else), which we write as $\mu(\pi i) = \pi e_i$.
Since $x_i \mapsto \epsilon_{i \emptyset}$ and $\epsilon_{i \emptyset}$ sends $\pi$ to $\pi i$, we may view $x_i \in \End(\alg{V}) \isom \Hom(\alg{V},\alg{V} \oplus \cdots \oplus \alg{V})$ as the column $i$th unit vector, identifying $\alg{V}$ with the $i$th summand. Likewise since
$y_i \mapsto \epsilon_{\emptyset i}$ and $\epsilon_{\emptyset i}$ sends $\pi i$ to $\pi$ (and acts as zero on sequences not ending with $i$), we may view $y_i \in \End(\alg{V}) \isom \Hom(\alg{V} \oplus \cdots \oplus \alg{V}, \alg{V})$ as the row $i$th unit vector.

\forget
This can be exploited in two ways: on one hand, $\End(\alg{V}) \isom \Hom(\alg{V}, \alg{V} \oplus \cdots \oplus \alg{V})$ is a ``column'' direct sum of $m$ copies of $\End(\alg{V})$, namely $\varphi \mapsto \mu \circ \varphi$; and on the other hand $\End(\alg{V}) \isom \Hom(\alg{V} \oplus \cdots \oplus \alg{V}, \alg{V})$ is a ``row'' direct sum of $m$ copies of $\End(\alg{V})$, each acting on a single component, namely $\varphi \mapsto \varphi \circ \mu^{-1}$. Since $x_i \mapsto \epsilon_{i\emptyset}$ and $\epsilon_{i \emptyset}$ sends $\pi$ to $\pi i$, it acts as the identification of $\alg{V}$ with the $i$th component, and consequently $x_i$ can be viewed as the column $i$th unit vector. Similarly, since $y_i \mapsto \epsilon_{\emptyset i}$ and $\epsilon_{\emptyset i}$ sets $\pi i$ to $\pi$ (and acts as zero on sequences not ending with $i$), $y_i$ is the map $(\pi_0,\dots,\pi_{m-1}) \mapsto \pi_i$, and can be viewed as the row $i$th unit vector.
\forgotten

We say that a matrix $x \in \rcfin$ is {\bf{rectangular $\mm$-recurrent}} if there are natural divisors $m^i, m^j \divides \mm$ and a rectangular matrix $a \in \M[m^i \times m^j](F)$ such that $x = a \oplus a \oplus \cdots$.
Although rectangular matrices cannot be multiplied, their recurrent sums are row- and column- finite matrices, which can be multiplied. For example, whereas $x_i$ is the recurrent $m \times 1$ unit vector, it can also be viewed as a recurrent $m^2 \times m$ matrix, and then $x_{i_0}x_{i_1}$ is the unit vector $e_{i_1i_0}$ of size $m^2$. More generally $x_{i_0}\cdots x_{i_{t-1}} y_{j_{s-1}} \cdots y_{j_{0}}$ is now mapped to the matrix unit $e_{i_{t-1}\cdots i_0, j_{s-1} \cdots j_0} \in \M[m^t \times m^s](F)$.

We thus proved:
\begin{thm}\label{recrec}
	$L_F(1,m)$ is isomorphic to the subalgebra of $\rcfin$ composed of rectangular $m^{\infty}$-recurrent matrices.
\end{thm}
This is an isomorphism of graded algebras, and the homogeneous component of degree zero is \Cref{recurr}.

\Tref{recrec} leads to a generalization of the Leavitt algebra, similar to the one given above for deep matrices: $x \in \rcfin$ is {\bf{rectangular $\nn$-recurrent}} if there are natural divisors $n, n' \divides \nn$ and a rectangular matrix $a \in \M[n \times n'](F)$ such that $x = a \oplus a \oplus \cdots$. The algebra of rectangular $\nn$-recurrent matrices should rightfully be denoted $L_F(1,\nn)$.

\subsection{The conditions $A^m \isom A$ and $\M[m](A) \isom A$}\label{ssec:LPA}

Write $\mm = m^{\infty}$. If an algebra~$A$ satisfies $A^m \isom A$ (as a module) then $\M[m](A) \isom A$ (as unital algebras). The simple Leavitt algebra $L_F(1,m)$ is a minimal solution to $A^m \isom A$, while its zero degree component $\M[\mm](F) = L_F(1,m)_0$ is a minimal solution to $\M[m](A) \isom A$. Any algebra containing $\M[\mm](F)$ but not $L_F(1,m)$ is a nonmodular solution to $\M[m](A) \isom A$.

Leavitt algebras demonstrate the difference between the two equations in the title of this subsection in another manner. By \cite[Theorem~3.10]{Lpa}, $L_F(1,m)$ solves $\M[d](A)\isom A$ if and only if $\gcd(d,m-1) = 1$. Thus, if $\gcd(d,m-1) = 1$, we have that $\M[d^{\infty}](F) \hra L_F(1,m)$.
\begin{cor}
	The algebra $L_F(1,m)$ contains the supernatural matrix algebra $\M[p^{\infty}](F)$ for almost every prime $p$; while its zero component $\M[\mm](F)$ contains almost no finite matrix algebra $\M[p](F)$.
\end{cor}
(The claim holds for every~$p$ prime to $m(m-1)$).

\subsection{Elementary Grading}

Famously, the equation
\begin{equation}\label{niceisom}
	\M[m](F) \tensor L \isom L,
\end{equation}
holds for
the Leavitt algebra $L = L_F(1,m)$ by defining $e_{ij} \tensor \ell \mapsto x_i \ell y_j$. This is a graded isomorphism when~$L$ is given the standard grading as above, and $\M[m](F)$ is graded trivially. The zero homogeneous component is $\M[m](F) \tensor L_F(1,m)_0 \isom L_F(1,m)_0$, as expected in light of \Tref{same}.

But there are other gradings of $L$. Let $G$ be any group, and $g_0,\dots,g_{m-1} \in G$ be arbitrary elements (we preserve the additive notation, although $G$ could well be nonabelian). Setting
$$\deg(x_i) = g_i, \qquad \deg(y_i) = -g_i$$
respects the defining relations, and induces a $G$-grading of $L$. In order for~\eq{niceisom} to preserve grading, we must grade $\M[m](F)$ so that $\deg(e_{ij}) = \deg(x_i y_j)$, namely $\deg(e_{ij}) = g_i- g_j$. This is called an {\bf{elementary}} grading of the matrix algebra (see~\cite{BZ}). The gradings of both algebras induce the decompositions into homogeneous components $L = \bigoplus_g L_g$ and $\M[m](F) = \bigoplus_g \M[m](F)_g$, and the homogeneous component at the identity provides us with an isomorphism
$$\bigoplus_g (\M[m](F)_g \tensor L_{-g}) =  (\M[m](F) \tensor L)_0 \isom L_0$$

\begin{exmpl}
	Fix $m = 2$, and consider the $\Z$-grading on $L = L_F(1,2)$ by $\deg(x_0) = - \deg(y_0) = 1$ and $\deg(x_1) = -\deg(y_1) = 2$. With this grading, the components of $L = \bigoplus L_k$ satisfy $L_k \isom \smat{L_k}{L_{k+1}}{L_{k-1}}{L_k}$ for every $k \in \Z$.
\end{exmpl}

Thus one is led to consider
elementary grading of the supernatural matrix algebra $\M[\nn](F)$, where $\nn$ is an arbitrary supernatural number. As before, let $1=n_0 \divides n_1 \divides n_2 \divides \cdots$ be divisors converging to $\nn$.

Let $G$ be a group. An elementary $G$-grading of $\M[\nn](F) = \dlim \M[n_t](F)$ is a $G$-grading which induces an elementary grading on each of the subalgebras $\M[n](F) \sub \M[\nn](F)$.
\begin{prop}
	Every elementary grading of $\M[\nn](F)$ is defined by
	$$\deg(e_{i_{t-1}\cdots i_0, i'_{t-1}\cdots i'_0}) = h_{0,i_0} + h_{1,i} + \cdots + h_{t-1,i_{t-1}} - h_{t-1,i_{t-1}} - \cdots - h_{0,i_0}$$
	where $h_{t,j} \in G$ are arbitrary elements ($t \geq 0$, $0 \leq j < \frac{n_{t+1}}{n_t}$).
\end{prop}
\begin{proof}
	By assumption, for every $t \geq 0$ there are $n_t$ elements $g_{t,u} \in G$ such that $\deg(e_{uv}) = g_{t,u} -g_{t,v}$, where $u,v$ are $\nn$-adic numbers of length~$t$. By the identification relation \eq{dmrel3}, we must have that $g_{t+1,iw}  -g_{t+1,iw'} = \deg(e_{iw,iw'}) = \deg(e_{w,w'}) = g_{t,w} - g_{t,w'}$ for every $i$ and every $w,w'$ of length $t$. Taking $w' = 0$, we have that
	$g_{t+1,iw} = g_{t,w} + h_{t,i}$ where $h_{t,i} = - g_{t,0} + g_{t+1,i0}$. By induction, it follows that $g_{t,i_{t-1}\cdots i_0} = h_{0,i_0} +  h_{1,i}+ \cdots+ h_{t-1,i_{t-1}}$, and the grading follows.
\end{proof}
\section{Simple presentations}
We are after the simple presentations of the supernatural matrix algebra $A = \M[\nn](F)$.
A first and most basic example of a simple module has been given in Section 1. One notices that this example is in fact a \textit{locally simple} as will be defined later. The most important result of this section is that in contrary to (finite) matrix algebras, supernatural matrix algebras admits uncountably many nonisomorphic pairwise simple modules.

Let $\alg{A}$ be any algebra in $\CC{\omega}$.
\begin{rem}\label{ldsr}
	Write $\alg{A} = \dlim (A_n,\varphi_n)$. For each $n$ let $V_n$ be a module of $A_n$. Let $f_n \co V_n \ra V_{n+1}$ be functions satisfying $f_n(ax) = \varphi_n(a)f_n(x)$ for all $a \in A_n$ and $x \in V_n$. Let $\alg{V} = \dlim(V_n,f_n)$. Then $\alg{V}$ is a module over $\alg{A}$.
\end{rem}
Indeed, the action of the $A_n$ on $V_n$ extends to an action of $\alg{A}$ on $\alg{V}$ unambiguously.

\begin{defn}
	A module $\alg{V}$ over $\alg{A}$ is {\bf{locally finite}} if there is a presentation $\alg{A} = \dlim A_n$ and there are finite dimensional modules $V_n$ over $A_n$, and compatible maps $f_n \co V_n \ra V_{n+1}$, such that $\alg{V} = \dlim (V_n,f_n)$. Similarly, for every property $P$, we say that $\alg{V}$ is locally-$P$ if $\alg{V}$ is locally finite and for every $n$, $V_n$ has the property $P$.
\end{defn}

\begin{prop}
	A locally simple module is simple.
\end{prop}
\begin{proof}
	It is equivalent to prove that the action of $A$ on $V$ is transitive. Let $v_1,v_2\in V$. They come from some $V_n$. Since $V_n$ is simple over $A_n$ there exists some $a\in A_n$ such that $av_1=v_2$.
\end{proof}

\begin{construction}\label{alldsr}
	The locally simple modules for $\M[\nn](F)$ are described as follows:
	\begin{enumerate}
		\item Let $1 = n_0 \divides n_1\divides n_2 \divides \cdots$ be a sequence such that $\nn = \lcm(n_i)$, and assume $\M[\nn](F) = \dlim \M[n_i](F)$ with the standard embeddings.
		\item Set $W_i = F^{n_i/n_{i-1}}$ and $V_i = W_1 \tensor \cdots \tensor W_i$, so that $\dim_F(V_i) = n_i$.
		\item Let $\alpha_i \in W_i$ be nonzero elements.
		\item Embed $V_i \ra V_{i+1}$ by $f_i \co x \mapsto x \tensor \alpha_i$. (These maps are compatible in the sense of \Rref{ldsr}).
	\end{enumerate}
	Then $\alg{V} = \dlim (V_i,f_i)$ is a locally simple, and every locally simple module is obtained in this way.
\end{construction}

This follows from the definition. Notice that if $(n_{k_i})$ is a subsequence of $(n_k)$, then there are more locally simple modules associated with $(n_{k_i})$ than for $(n_k)$, because in passing from $n_{k_i}$ to $n_{k_{i+1}}$ we may choose vectors in $F^{n_{k_{i+1}}/n_{k_i}}$ which are not tensor products from $F^{n_{k_{i+1}}/n_{k_{i+1}-1}} \tensor \cdots \tensor F^{n_{k_i+1}/n_{k_i}}$.

\begin{rem}
	The module $V$ of \Rref{originalV} results from taking each $\alpha_i \in W_i$ to be the first vector in the standard basis.
\end{rem}

\begin{prop}
	Two modules $\alg{V}$ and $\alg{V}'$ defined in \Tref{alldsr} over the same chain of naturals, are isomorphic if and only if for every large enough $i$, $\alpha_i$ and $\alpha_i'$ are linearly dependent.
\end{prop}
\begin{proof}
	$\Leftarrow$ This direction is clear.
	
	$\Rightarrow$ Assume there is an isomorphism $\varphi:\alg{V}\to \alg{V'}$ and choose an element $0\ne x\in \alg{V}$. Then $\ann(x)=\ann(\varphi(x))$.
	Choose some $n$ such that $x,\varphi(x)$ comes from $V_n$. With some abuse of notation, we denote their origin by $x,\varphi(x)$ too. Then $\ann_{A_n}(x)=\ann_{A_n}(\varphi(x))$, and for all $m>n$, $\ann_{A_m}(x\tensor \alpha_{n+1} \tensor \cdots \tensor \alpha_{m})=\ann_{A_m}(\varphi(x)\tensor \alpha_{n+1}' \tensor \cdots \tensor \alpha_{m}')$. It is known that two vectors have the same annihilator if and only if they are linearly dependent. So we are done. 
\end{proof}
\begin{cor}
	For every field $F$ there are uncountably many nonisomorphic locally simple modules over $M_{\nn}(F)$.  
\end{cor}
\begin{cor}\label{lsm}
	Let $\alg{V}$ be locally finite module over $M_{\nn}(F)$. Then $\alg{V}$ is locally simple if and only if there is a cofinal subset of the divisors of $\nn$ such that for every $x\in \alg{V}$ and every $n$ in this subset $\ann_{M_n(F)}(x)$ has co-dimension 1.  
\end{cor}

We construct a simple module over $\M[p^{\infty}](F)$ which is not locally simple. The idea is based upon the observation that if $\alg{V} = \dlim (V_n,f_n)$ is chosen so that for every $0 \neq x \in V_n$ we have that $f_n(V_n) \sub A_{n+1}f_n(x)$, then $\alg{V}$ is simple, and that this is a weaker assumption than $V_n$ being simple over $A_n$ (or $V_{n+1}$ being simple over $A_{n+1}$).


Fix a parameter $\ell$.

\begin{defn}
	Let $W$ be an arbitrary vector space.
	We define the rank of an element in $W \tensor F^\ell$ by $$\rank(\sum x_i \tensor e_i) = \dim \span\set{x_1,\dots,x_\ell}$$
	where $e_1,\dots,e_\ell$ is the natural basis of $F^\ell$.
\end{defn}

This will be used in the following context. Let $W$ be a vector space. Notice that $W \tensor F^\ell$ is a module over $\End(W)$ by the action $a(t\tensor e) = at \tensor e$.
\begin{prop}\label{p-trans}
	Let $x \in W \tensor F^\ell$. If $\rank(x) = \ell$ then $\End(W)\cdot x = W \tensor F^\ell$.
\end{prop}
\begin{proof}
	Write $x = \sum x_i \tensor e_i$ for $x_i \in W$. By assumption the $x_i$ are linearly independent, so there is $a \in \End(W)$ mapping them to any suitable vectors.
\end{proof}

\begin{prop}\label{46-15}
	Let $T \co W \ra U$ be a linear transformation. For every $x \in W \tensor F^\ell$, $\rank((T\tensor 1)(x)) \leq \rank(x)$.
\end{prop}
\begin{proof}
	Write $x = \sum_i x_i \tensor e_i$ and let $X = \span\set{x_1,\dots,x_\ell} \sub W$.
	By definition $\rank(x) = \dim X$, so that
	$$\rank((T\tensor 1)(x)) = \rank(\sum_i T(x_i)\tensor e_i) = \dim \span\set{T(x_1),\dots,T(x_\ell)} = \dim T(X) \leq \dim(X).$$ 
\end{proof}

Let $$\phi \co F^\ell \ra F^p \tensor F^\ell$$ be a vector space mapping, such that the image of every nonzero element has rank $\ell$. Necessarily, $\ell \leq p$.
\begin{exmpl}
	\begin{enumerate}
		\item If $E/F$ is a field extension of dimension $\ell$, then the embedding $E \hra \M[\ell](F) \hra F^p \tensor F^{\ell}$ defines a map with the desired properties.
		\item If $2\ell-1 \leq p$, the map can be defined by $\(\begin{matrix}\alpha_1 \\ \vdots \\ \alpha_{\ell}\end{matrix}\)
		\mapsto
		\(\begin{matrix}\alpha_1 & 0 & 0 & \cdots & 0 \\ \alpha_2 & \alpha_1 & 0 & \cdots & 0 \\ \alpha_3 & \alpha_2 & \alpha_1 &  \cdots & 0
			\\ \cdots & \cdots & \cdots & \cdots & \cdots
			\\ \alpha_{\ell} & \alpha_{\ell-1} & \alpha_{\ell-2} & \cdots & \alpha_1
			\\ 0 & \alpha_{\ell} & \alpha_{\ell-1} & \cdots & \alpha_2
			\\ 0 & 0 & \alpha_{\ell}  & \cdots & \alpha_3
			\\ \cdots & \cdots & \cdots & \cdots & \cdots
			\\ 0 & 0 & 0  & \cdots & \alpha_{\ell}
		\end{matrix}\)$.
	\end{enumerate}
\end{exmpl}

\begin{rem}
	The condition $2\ell-1 \leq p$ can always be obtained by replacing $p$ by a power $p^t$, as $\M[p^{\infty}](F) = \M[(p^t)^{\infty}](F)$.
\end{rem}

Consider the map $\ \id \tensor \phi \co W \tensor F^\ell \ra W \tensor (F^p \tensor F^\ell)$.
\begin{prop}\label{42:14}
	For every $0 \neq x  \in W\tensor F^\ell$, $\rank((\id \tensor \phi)(x)) = \ell$.
\end{prop}
\begin{proof}
	Write $x = \sum_i x_i \tensor e_i$, $x_i \in W$. Since some $x_i$ are nonzero, there is a functional
	$\ell \co W \ra F$ such that  $\sum \ell(x_i)e_i = (\ell \tensor 1)(x) \neq 0$. By assumption,
	$\phi(\sum \ell(x_i)e_i) \in F^p \tensor F^\ell$ has rank~$\ell$. The following diagram commutes:
	$$\xymatrix@C=40pt{
		W \tensor F^{\ell} \ar@{->}[r]^{1 \tensor \phi} \ar@{->}[d]^{\ell \tensor 1} & W \tensor F^p \tensor F^\ell \ar@{->}[d]^{\ell \tensor 1 \tensor 1} \\
		F^{\ell}\ar@{->}[r]^{\phi} & F^p \tensor F^{\ell}
	}
	$$
	Now, by \Pref{46-15},
	$$\ell = \rank(
	(\phi \circ (\ell \tensor 1))
	(x)) = \rank(
	((\ell \tensor 1 \tensor 1) \circ (1\tensor \phi))
	(x)) \leq
	\rank(
	(1\tensor \phi)(x)).$$

\end{proof}

\begin{thm}
	Let $A_n = \M[p](F)^{\tensor n}$ and let $V_n = (F^p)^{\tensor n} \tensor F^\ell$ be an $A_n$-module, with the action $a \cdot (t \tensor e) = at \tensor e$.
	
	Let $f_n \co V_n \ra V_{n+1}$ be the map $f_n = \id_n \tensor \phi$, where $\id_n$ is the identity on $(F^p)^{\tensor n}$.
	
	Then $\alg{V} = \dlim(V_n,f_n)$ is a simple module over $\alg{A} = \dlim A_n =\M[p^{\infty}](F)$, which is not locally simple.
\end{thm}
\begin{proof}
	To show that $\alg{V}$ is a module we verify the condition $f_n(a\cdot y) = \varphi_n(a)\cdot f_n(y)$ of \Rref{ldsr} (for $ a\in A_n$ and $y \in V_n$), where $\varphi_n(a) = a \tensor 1$. We take $y = x \tensor \alpha$ where $y \in (F^p)^{\tensor n}$ and $\alpha \in F^\ell$. By definition we have that
	$$f_n(a\cdot y) = (\id \tensor \phi)(ax \tensor \alpha) = ax \tensor \phi(\alpha)
	= (a\tensor 1) \cdot (x \tensor \phi(\alpha))= \varphi_n(a) \cdot f_n(y).$$
	
	To prove that $\alg{V}$ is simple, we observe that for every $0 \neq x \in V_n$, $f_n(x) \in (F^p)^{\tensor(n+1)} \tensor F^\ell$ has rank~$\ell$ by \Pref{42:14}. Since $A_{n+1} = \End((F^p)^{\tensor(n+1)})$, $A_{n+1}f_n(x) = (F^p)^{\tensor(n+1)} \tensor F^\ell = V_{n+1}$ by \Pref{p-trans}. Therefore $\alg{A}x = \alg{V}$.
	
	For a similar reason $\alg{V}$ cannot be locally simple, as the co-dimension of $(\ann_{A_n}(x))$ is at least~$\ell$, for any $x\in \alg{V}$, whereas this dimension would be $1$ for a locally simple module by \Cref{lsm}.
\end{proof}

\begin{rem}
	The same construction works for every supernatural number $\nn$ (taking $\phi_n \co F^\ell \ra F^{p_n} \tensor F^{\ell}$ instead of a fixed $\phi$, with $\nn = \prod p_n$).
	
	Therefore, every algebra of supernatural matrices has a simple module which is not the limit of simple modules.
\end{rem}

We sum up some connections between basic classes of modules: (Limits here are in the category of pairs, (ring, module), and the presentation of $A = \M[\nn](F)$ as a limit over $\omega$ is fixed.)
$$
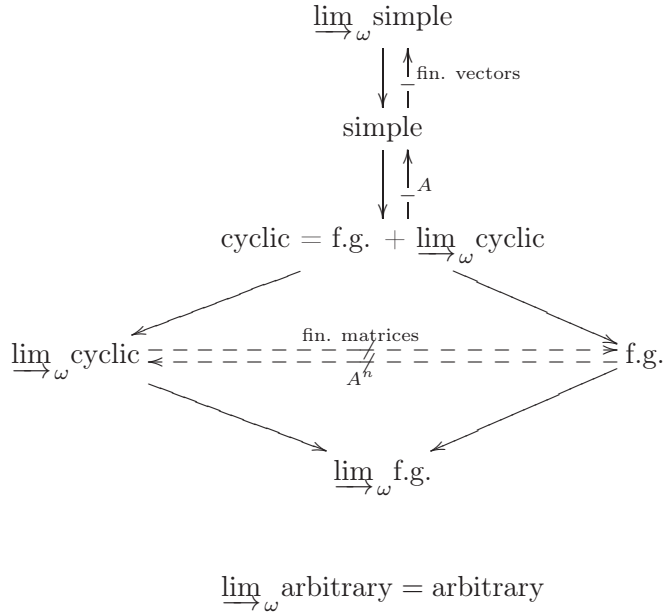
\xymatrix{
	{} & \dlim_{\omega}\mbox{simple} \ar@{->}[d] & {} \\
	{} & \mbox{simple} \ar@{->}[d]  \ar@<-2ex>@{->}[u]|(0.4){-}_{\mbox{\tiny{fin. vectors}}}
	& {} \\
	{} & \mbox{cyclic = f.g. + $\dlim_{\omega}$\mbox{cyclic}} \ar@{->}[dl] \ar@{->}[dr]
	\ar@<-2ex>@{->}[u]|(0.4){-}_{\mbox{\tiny{$A$}}} & {} \\
	\dlim_{\omega}\mbox{cyclic} \ar@{->}[dr] \ar@<0.5ex>@{-->}[rr]|(0.5){\not{}}^{\mbox{\tiny fin. matrices}}
	\ar@<-0.5ex>@{<--}[rr]|(0.5){\not{}}_{\mbox{\tiny $A^n$}} & {} & \mbox{f.g.} \ar@{->}[dl] \\
	{} & \dlim_{\omega}\mbox{f.g.}  & {} \\	{} & \dlim_{\omega}\mbox{arbitrary} = \mbox{arbitrary}& {}
}$$
\begin{proof}
	We will prove the directions which are not trivial.
	\begin{itemize}
		\item cyclic= f.g+ $\dlim_{\omega}$\mbox{cyclic}:
		
		Let $\alg{V}$ be a cyclic module over $A$. Then obviously $\alg{V}$ is f.g. In addition, there is some $x\in \alg{V}$ such that $\alg{V}= Ax$. So $\alg{V}=\dlim_{\omega}M_n(F)x$ a direct limit of cyclic modules over $M_n(F)$.
		
		On the other hand, assume that $\alg{V}$ is f.g. + $\dlim_{\omega}$\mbox{cyclic}. Denote by $\{x_1,...,x_n\}$ a finite set of generators. There is some $n$ such that $x_1,...,x_n\in V_n$. $V_n$ is cyclic so it can by generated by some $x$. Thus, $V$ is generated by $x$.
		\item  $\dlim_{\omega}\mbox{cyclic}\nrightarrow\mbox{f.g.}$: 
		
		Let $\alg{V}_n=M_n(F)$ with the embeddings: $a\to a\oplus 0$. Then $V_n$ is clearly a cyclic module over $M_n(F)$ and $V= \dlim_{\omega}V_n$ is a module over $M_{\nn}(F)$ consists of finite matrices. $V$ cannot be finitely generated, because that the ranks of matrices in $V$ would be bounded by the ranks of the generators. But clearly the rank are unbounded.
		\item  $\dlim_{\omega}\mbox{cyclic}\nleftarrow\mbox{f.g.}$
		
		Take $A^n$ for some $n>1$. This is clearly a f.g module. If it was a direct limit of cyclic modules then by the first equality we would get that it is cyclic itself. 	The fact that the free module $A^n$ (for any $n>1$) is not cyclic follows from $\M[\nn](F)$ being Dedekind-finite, by \Rref{limIBN}.  
	\end{itemize}
\end{proof}

\begin{rem}\label{limIBN}
	The direct limit of IBN rings is IBN.
\end{rem}
\begin{proof}
	Let $A = \dlim A_i$, where each $A_i$ is IBN. Suppose $A^n=A^m$. Take generators $x_r$ for the left-hand side, and $y_s$ for the right-hand side. Each $x_r$ can be expressed as a linear combination of the $y_s$, and for $i$ large enough, the coefficients all come from $A_i$, so $\sum_r A_i x_r \sub \sum_s A_i y_s$. For the same reason, if $i$ is large enough, we have an equality. But then, because of linear independence, $n=m$ (since this $A_i$ is IBN).
\end{proof}

\ifCM 

\else 

\fi 
\end{document}